\documentclass[12pt,reqno]{amsart}
\usepackage{amssymb,amsmath,dsfont}
\usepackage{pgfplots}
\pgfplotsset{compat=1.14}
\usepackage{enumitem}
\usepackage{calc}

\title[Small ball probabilities for the passage time]{Small ball probabilities for the passage time in planar first-passage percolation}

\date{}

\author{Dor Elboim}
\address{Dor Elboim\hfill\break
    Department of Mathematics,
    Stanford University,
    California, United States.}
\email{dorelboim@gmail.com}

\usepackage[margin=1in]{geometry}
\usepackage[all]{xy}
\usepackage{amssymb}
\usepackage{amsfonts}
\usepackage{amsthm}
\usepackage{amsmath}
\usepackage{hyperref}
\usepackage{mathtools}
\usepackage{url}
\usepackage{xcolor}
\usepackage{dsfont}
\usepackage{pgfplots}
\usepackage{comment}

\newtheorem{thm}{Theorem}[section]

\newtheorem{lem}[thm]{Lemma}  

\newtheorem{cor}[thm]{Corollary}

\mathtoolsset{showonlyrefs}

\numberwithin{equation}{section}

\begin{document}
\begin{abstract}
   We study planar first-passage percolation with independent weights whose common distribution is supported in $(0,\infty)$ and is absolutely continuous with respect to Lebesgue measure. We prove that the passage time from $x$ to $y$ denoted by $T(x,y)$ satisfies 
   \begin{equation}
       \max _{a\ge 0} \mathbb P \big(  T(x,y)\in [a,a+1] \big) \le \frac{C}{\sqrt{\log \|x-y\|}},
   \end{equation}
 answering a question posed by Ahlberg and de la Riva \cite{ahlberg2023being}. This estimate recovers earlier results on the fluctuations of the passage time by Newman and Piza \cite{newman1995divergence}, Pemantle and Peres \cite{pemantle1994planar} and Chatterjee \cite{chatterjee2019general}.
\end{abstract}

\maketitle
\section{Introduction}
First-passage percolation is a model for a random metric space, formed by a random perturbation of an underlying base space. Since its introduction by Hammersley--Welsh in 1965~\cite{HammersleyWelsh}, it has been studied extensively in the probability and statistical physics literature. We refer to \cite{Kesten:StFlour} for general background and to \cite{50years} for more recent results.


We study first-passage percolation on the square lattice $(\mathbb Z^2,E(\mathbb{Z}^2))$, in an independent and identically distributed (IID) random environment. The model is specified by a \emph{weight distribution $G$}, which is a probability measure on the non-negative reals. It is defined by assigning each edge $e\in E(\mathbb{Z}^2)$ a random passage time $t_e$ with distribution $G$, independently between edges. Then, each finite path $p$ in $\mathbb Z ^2$ is assigned the random passage time
\begin{equation}\label{eq:time of a path}
    T(p):=\sum _{e\in p} t_e,
\end{equation}
yielding a random metric $T$ on $\mathbb Z^2$ by setting the \emph{passage time between $u,v\in \mathbb Z ^2$} to
\begin{equation}\label{eq:defT}
    T(u,v):=\inf_{p} T(p),
\end{equation}
where the infimum ranges over all finite paths connecting $u$ and $v$. Any path achieving the infimum is termed a \emph{geodesic} between $u$ and $v$. A unique geodesic exists when $G$ is atomless and will be denoted $\gamma(u,v)$. The focus of first-passage percolation is the study of the large-scale properties of the random metric $T$ and its geodesics.

Ahlberg and de la Riva \cite[Equation (11) and Section 7]{ahlberg2023being} raised the problem of proving that the passage time between far away points is unlikely to be in any given interval of constant length. More precisely, they conjectured that for all $C>0$ one has 
\begin{equation}
    \max _{a>0} \mathbb P \big( T(0,(n,0))\in [a,a+C] \big) \to 0, \quad n\to \infty .
\end{equation}
In our main result below we establish a quantitative version of this conjecture for the class of absolutely continuous weight distributions.

\begin{thm}\label{thm:small}
    Suppose that the weight distribution $G$ is absolutely continuous. Then, there exists a constant $C>0$ depending on $G$ such that for any $x\in \mathbb Z ^2$ with $\|x\|\ge 2$
    \begin{equation}\label{eq:constant number of edges}
    \max _{a \ge 0}  \mathbb P \big( T(0,x)\in [a,a+1] \big)   \le \frac{C}{\sqrt{\log \|x\|}}.
\end{equation}
\end{thm}

\subsection{Previous works}
Pemantle and Peres \cite{pemantle1994planar} proved an estimate similar to \eqref{eq:constant number of edges} when the weight distribution $G$ is exponential. The proof relies on the memoryless property of the exponential distribution. We are not aware of a similar result for any other weight distribution.

The estimate in \eqref{eq:constant number of edges} is closely related to anti-concentration bounds for the passage time. The first result of this kind is by Newman and Piza \cite{newman1995divergence} who proved that
\begin{equation}\label{eq:var}
    \text{Var}(T(0,x))\ge c\log \|x\|,
\end{equation}
for a large family of weight distributions. Later, \cite{zhang2008shape} and \cite{auffinger2013differentiability} extended the result to a class of weight distributions in which there is a heavy atom at the essential infimum of the distribution, and to directions that are outside of the percolation cone (see \cite[Section~1.1]{auffinger2013differentiability} for more details).

Chatterjee \cite[Theorem~2.6]{chatterjee2019general} proved the stronger result that the fluctuations of the passage time are at least of order $\sqrt{\log \|x\|}$ by showing that there exists $c>0$ such that for all $x\in \mathbb Z ^2$ and all $a<b$ with $b-a\le c\sqrt{\log \|x\|}$,
\begin{equation}\label{eq:fluc}
    \mathbb P \big( T(0,x) \notin [a,b]  \big)\ge c.
\end{equation}
This is done for absolutely continuous weight distributions whose density decays faster than any polynomial, and is of the form $e^{-V}$, where $V$ is smooth and all its derivatives grow at most polynomially. This was generalized to a wider class of weight distributions by Damron, Hanson, Houdr{\'e} and Xu \cite{damron2020lower} and independently by Bates and Chatterjee \cite{bates2020fluctuation}.

Observe that that both \eqref{eq:var} and \eqref{eq:fluc} follow from Theorem~\ref{thm:small} for the class of absolutely continuous weight distributions. 

Finally, let us remark that the perturbation argument used in this paper is similar to that of Chatterjee \cite{chatterjee2019general} with one notable difference. In \cite{chatterjee2019general} the author considers a single perturbation of the weights while in here we are continuously perturbing the weights. We then show that this continuous perturbation quickly pushes the passage time out of any given interval. To establish this, we use a slightly more robust perturbation lemma that holds in greater generality and captures small probabilities (see Section~\ref{sec:mermin} below).

\subsection{Acknowledgments}
We thank Barbara Dembin, Ron Peled, Daniel Ahlberg and Daniel de la Riva for fruitful discussions about this problem. We thank the two referees of the paper for
their comments.

\section{A Mermin--Wagner type estimate}\label{sec:mermin}

The following lemma is the main technical tool required for the proof of the main theorem. The lemma is a Mermin--Wagner type estimate and is taken from \cite[Lemma~2.12]{dembin2024coalescence}. A similar lemma was used also in \cite{dembin2025influence} and \cite{dembin2025noise}.

\begin{lem}\label{lem:MW}
  Suppose that $G$ is absolutely continuous distribution on $\mathbb R$. Then, there exist 
  \begin{itemize}
  \item  A Borel set $S_G$ contained in the closure of the support of $G$, with $G(S_G)=1$.
      \item Borel subsets $(B_{\delta })_{\delta>0}$ of $S_G$ with $\lim_{\delta\downarrow 0}G(B_\delta)=1$,
      \item For each $\tau \in \mathbb R$, a bijection $g_{\tau }:S_G\to S_G$.
  \end{itemize}
   such that the following holds:
  \begin{enumerate}
  \item The function $g_\tau (s)$ is increasing both in $s$ and in $\tau $.
      \item $g_0$ is the identity function on $S_G$ and for any $\tau _1,\tau _2\in \mathbb R$ we have $g_{\tau _1}\circ g_{\tau _2}=g_{\tau _1+\tau _2}$.
      \item For any $\tau \in[0,1]$ and $s\in B_\delta $ we have  $g_{\tau }(s)\ge s+\delta \tau $.
      \item For any integer $n\ge 1$, a vector $\tau\in \mathbb R ^n$ and a Borel set $A\subset\mathbb{R}^n$ we have
      \begin{equation}\label{eq:MW probability estimate}
             \mathbb P (X\in A) \le e^{\|\tau \|_2^2/2} \sqrt{ \mathbb P \big(  (g_{\tau _i}(X_i)) _{i=1}^n\in A \big)\mathbb P \big( (g_{-\tau_i}(X_i))_{i=1}^n \in A \big)  }
      \end{equation}
      where $X=(X_1,\dots ,X_n)$ is a vector of i.i.d.\ random variables with distribution $G $.
  \end{enumerate}
\end{lem}

\begin{proof}[Proof sketch]
     Let us briefly explain the main ideas in the proof of the lemma while the full argument is given in \cite{dembin2024coalescence}. First, we consider the case when $G=N(0,1)$ is the standard normal distribution. In this case the functions $g_\tau $ are simply taken to be $g_\tau (s):=s+\tau $. With this choice the inequality \eqref{eq:MW probability estimate} translates to
    \begin{equation}\label{eq:gauss}
             \mathbb P (X\in A) \le e^{\|\tau \|_2^2/2} \sqrt{ \mathbb P \big(  X+\tau \in A \big)\mathbb P \big( X-\tau  \in A \big)  }
      \end{equation}
     which follows from a simple application of Cauchy-Schwarz inequality (see \cite[Claim~2.13]{dembin2024coalescence}).
     
     For a general continuous distribution $G$, we consider the (unique) increasing function $h:\mathbb R \to S_G$ with $h(N)\sim G$ where $N$ is a standard normal variable. Then, we define the function $g_\tau$ by $g_\tau (s):=h(h^{-1}(s) +\tau )$ and prove that this function satisfies the requirements (the inequality \eqref{eq:MW probability estimate} follows immediately as it is the push forward of \eqref{eq:gauss} by $h$).

     Finally, let us note that the statement of  \cite[Lemma~2.12]{dembin2024coalescence} does not include parts (1) and (2) of Lemma~\ref{lem:MW}. However, these parts follow immediately from our definition of the perturbation function $g_\tau (s)=h(h^{-1}(s) +\tau )$.
\end{proof}

\section{Proof of Theorem~\ref{thm:small}}\label{sec:constant}

Let $x\in \mathbb Z ^2$ with $\|x\|$ is sufficiently large and let $n:=\|x\|$. Let $\Lambda _k$ be the set of edges with both endpoints in the annulus $[-2^{k+1},2^{k+1}]^2\setminus [-2^k,2^k]^2$ or one endpoint in the annulus and the other in $[-2^k,2^k]$. Let $k_0:=\lfloor \log _2\sqrt{n} \rfloor $ and $k_1:=\lfloor \log _2n \rfloor -1$.

For $r\in [-2,2]$ we define the perturbation function $\tau _r:E(\mathbb Z ^d)\to \mathbb R$ by
\begin{equation}
    \tau _r (e):= \frac{r}{2^k\sqrt{\log n}}, 
\end{equation}
for all $e\in \Lambda _k$ with $k_0\le k\le k_1$. We let $\tau _r(e)=0$ for edges outside of $\bigcup _{k=k_0}^{k_1}\Lambda _k$.

We define the modified environment by $t_{e,r}:=g_{\tau _r(e)}(t_e)$, where $g_\tau $ is from Lemma~\ref{lem:MW}. Note that the modified weights are still positive. We denote by $T_r(p)$ the weight a path $p$ in the modified environment and by $T_r(x,y)$ the passage time between $x$ and $y$ in the modified environment. In our proof we continuously vary $r$ and estimate the effect it has on the passage time $T(0,x)$.

It is important that for all $r\in [-1,1]$ we have 
\begin{equation}
    \|\tau _r\|_2^2 \le  \sum _{k=k_0}^{k_1} \frac{|\Lambda _k|}{4^k \log n} \le \frac{C(k_1-k_0)}{\log n} \le C 
\end{equation}
so that in Lemma~\ref{lem:MW} the factor $e^{\|\tau \|_2^2/2}$ remains constant.

We begin by showing that the weight of paths in scale $k$ is increased by this perturbation. To this end let $\mathcal P _k$ be the set of paths of length $2^k$ whose edges are in $\Lambda _k$ and let $\delta _0$ such that $G(B_{\delta _0})\ge 0.999$. 

\begin{lem}\label{lem:1}
For any $k_0\le k \le k_1$ and $ r\in [0,1]$ we have that 
\begin{equation}
\mathbb P \Big( \exists p\in \mathcal P _k, \  T_{r}(p)-T(p)\le \frac{\delta _0r}{2\sqrt{\log n}}  \Big) \le e^{-2^k}.    
\end{equation}
\end{lem}

\begin{proof}
    Let $k_0\le k\le k_1$ and $r\in [0,1]$. Recall that if $t_e\in B_{\delta _0}$ then 
    \begin{equation}
     t_{e,r}=g_{\tau _r(e)} (t_e)\ge t_e+\delta _0\tau _r(e)=t_e+\frac{\delta _0r}{2^k\sqrt{\log n}}.
    \end{equation}
     Thus, for a fixed path $p\in \mathcal P_k$ we have 
    \begin{equation}
    \begin{split}
\mathbb P \Big( T_{r}(p)-T(p)\le   \frac{\delta _0r}{2\sqrt{\log n}}  \Big) \le \mathbb P \big( \big| \big\{e\in p : t_e\in B_{\delta _0} \big\} \big| \le 2^k/2 \big) \le \\
\mathbb P \big( \text{Bin}(2^k,0.999) \le 2^k/2 \big) \le 8^{-2^k}.    
\end{split}
\end{equation}
Moreover, we have that $|\mathcal P_k| \le C4^{k}3^{2^k}$ as we can first choose the stating point of $p\in \mathcal P_k$ and then choose each step of the path. The lemma follows from a union bound over $p\in \mathcal P _k$ as long as $n$ is sufficiently large.
\end{proof}

\begin{cor}\label{cor:2}
    For any $k_0\le k\le k_1$, $s\in [-1,1]$ and $r\in [0,1]$ we have that 
\begin{equation}
\mathbb P \Big( \exists p\in \mathcal P _k, \  T_{s+r}(p)-T_s(p)\le \frac{\delta _0r}{2\sqrt{\log n}}  \Big) \le Ce^{-2^{k-1}}.    
\end{equation}
\end{cor}

\begin{proof}
The corollary follows from Lemma~\ref{lem:1} using Lemma~\ref{lem:MW}. Let us note that this will not be the main use of Lemma~\ref{lem:MW} and that this part can be done using different arguments.

The set $A\subseteq \mathbb R ^{\Lambda _k}$ from Lemma~\ref{lem:MW} will be 
\begin{equation}
    A:=\Big\{ (z_e)_{e\in \Lambda _k} :  \exists p\in \mathcal P _k , \ \sum _{e\in p} g _{\tau _{s+r}(e) }(z_e)-g_{\tau _s(e)}(z_e) \le \frac{\delta _0r}{2\sqrt{\log n}} \Big\}.
\end{equation}
By Lemma~\ref{lem:MW} and the fact that $\|\tau _s\|_2\le C$ we have that
\begin{equation}
\mathbb P ((t_e)_{e\in \Lambda _k} \in A) \le C\sqrt{\mathbb P \big( (g_{-\tau _s(e)}(t_e))_{e\in \Lambda _k} \in A \big) }.    
\end{equation}
It is easy to check (using part (2) of Lemma~\ref{lem:MW}) that the event on the right hand side of the last inequality is precisely the event of Lemma~\ref{lem:1} while the event on the left hand side is the event of the corollary.
\end{proof}

\begin{lem}\label{cor:1}
    There exists $C>0$ depending only on $G$ such that for any $a\ge 0$ we have 
    \begin{equation}
        \int _{-1}^1 \mathbb P \big(  T_r(0,x) \in [a,a+1] \big) dr \le \frac{C}{\sqrt{\log n}}.
    \end{equation}
\end{lem}

\begin{proof}
Let $r_0:= 8/(\delta _0\sqrt{\log n}) $  and let $R:=r_0\mathbb Z \cap [-1,1]$. Define the event
\begin{equation}
    \Omega :=\bigcap _{k=k_0}^{k_1} \bigcap _{s\in  R}  \Big\{ \forall p \in \mathcal P_k, \  T_{s+r_0} (p)-T_{s}(p) > \frac{4}{\log n} \Big\}.  
\end{equation}
Using Corollary~\ref{cor:2} and a union bound we obtain
\begin{equation}\label{eq:Omega}
\mathbb P (\Omega ^c )\le  C\sum _{k=k_0}^{k_1} \sum _{s\in R} e^{-2^{k-1}}   \le  Ce^{-c\sqrt{n}}.
\end{equation}

Next, we write 
\begin{equation}\label{eq:1}
        \int _{-1}^1 \mathbb P \big(  T_r(0,x) \in [a,a+1] \big) dr \le 2\,\mathbb P (\Omega ^c)+  \int _{-1}^1 \mathbb P \big( \Omega , \ T_r(0,x) \in [a,a+1] \big) dr.
    \end{equation}
 By \eqref{eq:Omega}, it suffices to bound the second term on the right hand side of \eqref{eq:1}. To this end, note that any path $p$ connecting $0$ to $x$ will have a subpath in $\mathcal P _k$ for any $k_0\le k\le k_1$. It follows that on $\Omega$ for any path $p$ connecting $0$ and $x$ and for all $s\in R$ we have $T_{s+r_0}(p)-T_s(p)\ge 4(k_1-k_0)/\log n \ge 2$. Hence, on $\Omega $ we have $T_{s+r_0}(0,x)-T_s(0,x)\ge 2$ for all $s\in R$. Thus, using also that $T_r(0,x)$ is increasing in $r$ we obtain that on $\Omega$ there is at most one element $s\in R$ for which $T_s(0,x)\in [a,a+1]$ and that $\big| \big\{r\in [-1,1] : T_r(0,x)\in [a,a+1] \big\}\big| \le 2r_0=16/(\delta _0\sqrt{\log n})$, where in here $|\cdot |$ denotes the Lebesgue measure of the set. Hence, using Fubini's theorem
    \begin{equation*}
        \int _{-1}^1 \mathbb P \big( \Omega , \ T_r(0,x) \in [a,a+1] \big) dr= \mathbb E \Big[ \mathds 1 _{\Omega }\big| \big\{r\in [-1,1] : T_r(0,x)\in [a,a+1] \big\}\big| \Big] \le \frac{16}{\delta _0\sqrt{\log n}}.
    \end{equation*}
    Substituting this estimate into \eqref{eq:1} finishes the proof of the lemma.
\end{proof}

We can now prove Theorem~\ref{thm:small}. 

\begin{proof}[Proof of Theorem~\ref{thm:small}]
    We claim that for any $r\in [0,1]$ and $a>0$ we have \begin{equation}\label{eq:mw}
        \mathbb P \big( T(0,x)\in [a,a+1] \big) \le C \sqrt{ \mathbb P \big( T_r(0,x)\in [a,a+1] \big) \mathbb P \big( T_{-r}(0,x)\in [a,a+1] \big)}.
    \end{equation}
    Unfortunately, this is not a direct use of Lemma~\ref{lem:MW} since $T(0,x)$ depends on infinitely many random variables while in Lemma~\ref{lem:MW} there are finitely many. Let us briefly explain how to overcome this issue using an approximation argument. For $R>n$ define the restricted passage time $T^R(0,x)$ exactly as in \eqref{eq:defT} but when the infimum is over paths restricted to stay in the box $[-R,R]^2$. Note that the restricted passage time depends only on the variables in the box $[-R,R]^2$. Similarly, we define the restricted passage times $T^R_r(0,x)$ and $T^R_{-r}(0,x)$ in the modified environments. Using Lemma~\ref{lem:MW} we obtain for all $R>n$
    \begin{equation}\label{eq:restricted}
        \mathbb P \big( T^R(0,x)\in [a,a+1] \big) \le e^{\|\tau _r\|^2/2} \sqrt{ \mathbb P \big( T^R_r(0,x)\in [a,a+1] \big) \mathbb P \big( T^R_{-r}(0,x)\in [a,a+1] \big)}.
    \end{equation}
    Next, recall that $\gamma (0,x)$ is the geodesic from $0$ to $x$. It is easy to check (and follows from, e.g., \cite[Proposition~4.4]{50years}) that for all $\epsilon >0$ there exists $R$ such that 
    \begin{equation}
     \mathbb P \big( T^R(0,x)\neq T(0,x) \big) \le \mathbb P \big( \gamma (0,x) \nsubseteq [-R,R]^2 \big)  \le \epsilon.    
    \end{equation}
    Similarly, for $R$ sufficiently large 
    \begin{equation}
     \mathbb P \big( T^R_r(0,x)\neq T_r(0,x) \big) \le \epsilon \quad \text{and}\quad \mathbb P \big( T^R_{-r}(0,x)\neq T_{-r}(0,x) \big) \le \epsilon.
    \end{equation}
    We can therefore replace the restricted passage times in \eqref{eq:restricted} with the non restricted passage times and this will change the probabilities by at most $\epsilon$. Taking $\epsilon $ to zero finishes the proof of \eqref{eq:mw}.
    
    Integrating \eqref{eq:mw} over $r\in [0,1]$ and using Cauchy-Schwartz inequality we obtain
    \begin{equation*}
        \mathbb P \big( T(0,x)\in [a,a+1] \big) \le C \sqrt{ \int _0^1\mathbb P \big( T_r(0,x)\in [a,a+1] \big) dr \int_0^1 \mathbb P \big( T_{-r}(0,x)\in [a,a+1] \big)dr}.
    \end{equation*}
    By Lemma~\ref{cor:1} the right hand side of the last equation is bounded by $C/\sqrt{\log n}$ which finishes the proof of the theorem. 
\end{proof}

\bibliographystyle{plain}
\bibliography{biblio}

\end{document}